\documentclass[12pt]{amsart}

\usepackage{fullpage}
\usepackage{amsmath}
\usepackage{amssymb}
\usepackage{amsthm}
\usepackage{amstext}
\usepackage{appendix}
\usepackage{perpage}
\usepackage{color}
\usepackage{tikz-cd}
\usepackage[colorlinks=true,linkcolor=blue, citecolor=blue]{hyperref}

\newtheorem{thm}{Theorem}[section]

\newtheorem{coro}[thm]{Corollary}
\newtheorem{prop}[thm]{Proposition}

\theoremstyle{definition}

\newtheorem{defi}[thm]{Definition}
\newtheorem{remark}[thm]{Remark}

\begin{document}

\title{Green functions and higher Deligne--Lusztig characters}

\author{Zhe Chen}

\address{Department of Mathematics, Shantou University, Shantou, China}

\email{zhechencz@gmail.com}

\begin{abstract}
We give a generalisation of the character formula of Deligne--Lusztig representations from the finite field case to the truncated formal power series case. Motivated by this generalisation, we give a definition of Green functions for these local rings, and we prove some basic properties along the lines of the finite field case, like a summation formula. Among the applications we show that the higher Deligne--Lusztig characters and G\'erardin's characters agree at regular semisimple elements. We also derive a generalisation of Braverman and Kazhdan's result on gamma functions for Deligne--Lusztig characters, with a more elementary argument.
\end{abstract}

\maketitle

\tableofcontents

\section{Introduction and notations}\label{section:Intro}

The family of representations introduced in Deligne and Lusztig's work \cite{DL1976} plays a crucial role in the representation theory of reductive groups over finite fields, since 1976. To compute the characters of these representations, there are roughly two steps involved: The first one is a formula expressing these characters in terms of Green functions of smaller reductive groups, and the second one is to compute the Green functions using generalised Springer theory and character sheaves (Lusztig--Shoji algorithm). 

\vspace{2mm} In this paper we give a study on the generalisation of the above first step, from finite fields to discrete valuation rings. There are four perspectives. In the first part, we prove a character formula for reductive groups over (quotients of) complete discrete valuation rings of positive characteristics, which expresses the characters of higher Deligne--Lusztig representations in terms of the traces of certain unipotent elements; see Theorem~\ref{theorem: character formula} and Theorem~\ref{theorem: character formula 2}. This generalises the character formula from the finite field case in a natural way. In the second part, motivated by this generalisation, we give a definition of Green functions in this ring setting; see Definition~\ref{definition: Green function}. Similar to the finite field case, these functions are defined on unipotent elements, and are independent of the choice of characters of the maximal torus (which are parameters of higher Deligne--Lusztig characters). We show that they enjoy some nice properties, such as a summation formula; see Corollary~\ref{coro: integer values of Green functions}, Corollary~\ref{coro: a product property}, Corollary~\ref{coro: a product property in inner product form}, and Proposition~\ref{prop: summation formula}. In the third part we focus on regular semisimple elements. Using the Green functions we establish a formula of higher Deligne--Lusztig characters on regular semisimple elements; see Theorem~\ref{thm: evaluation at regular semisimple elements}. This formula is independent of an integer parameter involved, and as an immediate application, we deduce that, at regular semisimple elements, the values of higher Deligne--Lusztig characters are the same with that of G\'erardin's characters (constructed in \cite{Gerardin1975SeriesDiscretes}), thus obtaining a variation of a prediction of Lusztig \cite{Lusztig2004RepsFinRings}; see Corollary~\ref{coro: coincidence at even levels} and Remark~\ref{remark: coincidence at odd levels}. In the last section, we study the gamma functions originally introduced by Braverman and Kazhdan in \cite{Braverman_Kazhdan_2000_GammaFunc} for general linear groups over finite fields. In the case of higher Deligne--Lusztig characters with $p$-constant parameters, we use the summation formula to prove that the associated gamma function is the same as the gamma function of the corresponding finite tori (see Proposition~\ref{prop: gamma function}); this gives a generalisation of \cite[9.3]{Braverman_Kazhdan_2000_GammaFunc} and \cite[5.3]{Braverman_Kazhdan_2000_gammaSh} for higher Deligne--Lusztig characters, by a different method.

\vspace{2mm} In the below we introduce some notations and describe some results with more details.

\vspace{2mm} Let $\mathcal{O}$ be a complete discrete valuation ring with a finite residue field $\mathbb{F}_q$, let $\pi$ be a fixed uniformiser of $\mathcal{O}$, and let $\mathbb{G}$ be a connected reductive group over $\mathcal{O}_r:=\mathcal{O}/\pi^r$ where $r$ is a fixed arbitrary positive integer. We want to study the complex smooth representations of reductive groups over $\mathcal{O}$, or equivalently, the complex representations of $\mathbb{G}(\mathcal{O}_r)$ (for all $r\in\mathbb{Z}_{>0}$). One approach for studying these representations for a general $r\geq 1$, on which this paper is based, is a geometric theory proposed by Lusztig in \cite{Lusztig1979SomeRemarks} (with some missing proofs given in \cite{Lusztig2004RepsFinRings} for $\mathrm{char}(\mathcal{O})=0$, which were generalised in \cite{Sta2009Unramified} for all characteristics using the Greenberg functor technique). This theory generalises the work in the $r=1$ case in \cite{DL1976}, so we call it the higher Deligne--Lusztig theory. In this paper we will be only working with $\mathrm{char}(\mathcal{O})>0$; in particular, we put $\mathcal{O}=\mathbb{F}_q[[\pi]]$.

\vspace{2mm} We recall the basic settings in \cite{Lusztig2004RepsFinRings}. Write $\mathbf{G}:=\mathbb{G} \times_{\mathrm{Spec}~\mathcal{O}_r}\mathrm{Spec}~\mathcal{O}^{\mathrm{ur}}_r$, where $\mathcal{O}^{\mathrm{ur}}:=\overline{\mathbb{F}}_q[[\pi]]$ is a maximal unramified extension of $\mathcal{O}$ and $\mathcal{O}^{\mathrm{ur}}_r:=\mathcal{O}^{\mathrm{ur}}/\pi^r$. Let $\mathbf{B}=\mathbf{T}\ltimes\mathbf{U}$ be the Levi decomposition of a Borel subgroup of $\mathbf{G}$. By Weil restriction, we can view $\mathbf{G}(\mathcal{O}^{\mathrm{ur}}_r)$, $\mathbf{B}(\mathcal{O}^{\mathrm{ur}}_r)$, $\mathbf{T}(\mathcal{O}^{\mathrm{ur}}_r)$, and $\mathbf{U}(\mathcal{O}^{\mathrm{ur}}_r)$ as the $\overline{\mathbb{F}}_q$-points of some algebraic groups $G$, $B$, $T$, and $U$, respectively, over $\overline{\mathbb{F}}_q$. The Frobenius map on $\overline{\mathbb{F}}_q/\mathbb{F}_q$ induces a geometric Frobenius endomorphism on $G$, such that $G^F\cong \mathbb{G}(\mathcal{O}_r)$. We only consider the case that $F(T)\subseteq T$. Let $L$ be the Lang isogeny associated to $F$, namely, $L(g)=g^{-1} F(g)$ for all $g\in G$. We fix a rational prime integer $\ell \nmid q$; we will be working with the $\overline{\mathbb{Q}}_{\ell}$-representations.

\vspace{2mm} The variety $L^{-1}(FU)\subseteq G$ admits a left $G^F$-action and a right $T^F$-action in a natural way. These actions induce a $(G^F,T^F)$-bimodule structure on the compactly supported $\ell$-adic cohomology groups $H_c^i(L^{-1}(FU),\overline{\mathbb{Q}}_{\ell})$. Let $H_c^i(L^{-1}(FU),\overline{\mathbb{Q}}_{\ell})_{\theta}$ be the isotypical space for a given $\theta\in\widehat{T^F}:=\mathrm{Hom}(T^F,\overline{\mathbb{Q}}_{\ell})$, then 
$$R_{T,U}^{\theta}:=\sum_{i} (-1)^i H_c^i(L^{-1}(FU),\overline{\mathbb{Q}}_{\ell})_{\theta}$$
is a virtual representation of $G^F$, referred to as a higher Deligne--Lusztig representation.

\vspace{2mm} Denote $G$ by $G_r$, then for each $i\in[0,r]\cap\mathbb{Z}$, there is a morphism of algebraic groups $\rho_{r,i}\colon G_r\rightarrow G_i$, called the reduction map, induced by the reduction modulo $\pi^i$. The morphism $\rho_{r,i}$ is surjective by the smoothness of $\mathbf{G}$; we denote its kernel, a normal closed subgroup of $G$, by $G^i$. In particular, we have $G^0=G$ (do not mix it with the identity component $G^o$). Similar notation applies to the closed subgroups of $G$ (like $B$, $U$, and $T$). 

\vspace{2mm} The above objects lead to the representation
$$R_{T,U,b}^{\theta}:=\sum_{i} (-1)^i H_c^i(L^{-1}(FU^{b,r-b}),\overline{\mathbb{Q}}_{\ell})_{\theta},$$
where $b\in [0,r]\cap\mathbb{Z}$ and $U^{b,r-b}:=U^b(U^{-})^{r-b}$ (here $U^-$ denotes the algebraic group corresponding to the opposite of $\mathbf{U}$). This construction was first studied in \cite{Chen_2017_InnerProduct}; note that it naturally generalises the representations studied in \cite{Lusztig1979SomeRemarks}, \cite{Lusztig2004RepsFinRings}, \cite{Sta2009Unramified}, \cite{Chen_2016_GenericCharSh}, and \cite{ChenStasinski_2016_algebraisation} (as clearly $R_{T,U}^{\theta}=R_{T,U,r}^{\theta}$). The $R_{T,U,b}^{\theta}$'s are the basic objects in this paper. We prove that (see Theorem~\ref{theorem: character formula})
\begin{equation*}
\begin{split}
\mathrm{Tr}(g,R_{T,U,b}^{\theta})=
&\frac{1}{|T^F|}\cdot \frac{1}{|(\mathrm{Stab}_{G}(s)^o)^F|} \cdot  \sum_{    \left\lbrace \substack{ h\in G^F\  \text{s.t.} \\ s\in {^h(T_1)^F} } \right\rbrace    } \sum_{\tau\in {^h(T^1)}^F}         {\theta}(s^h\cdot \tau^h) \\
&  \cdot  \mathrm{Tr}\left((u,\tau^{-1})  \mid \sum_{i}(-1)^i\cdot H_c^i(\mathrm{Stab}_G(s)^o\cap L^{-1}( {^h FU^{r-b,b} }),\overline{\mathbb{Q}}_{\ell})  \right),
\end{split}
\end{equation*}
where $g=su\in G^F$ denotes the Jordan decomposition. This suggests a definition of two-variable Green functions $Q_{T,U,b}^{G}(-,-)$ defined on some unipotent elements (see Definition~\ref{definition: Green function}), which admit the following summation property (see Proposition~\ref{prop: summation formula})
$$\sum_{u\in (\mathcal{U}_{G})^F}\sum_{\tau\in ((T^1))^F} Q_{T,U,b}^{G}(u,\tau)=|G^F/T_1^F|.$$ 
These Green functions allow us to re-write the above character formula as (see Theorem~\ref{theorem: character formula 2})
\begin{equation*}
\mathrm{Tr}(g,R_{T,U,b}^{\theta})=\frac{1}{|(\mathrm{Stab}_{G}(s)^o)^F|} \cdot  \sum_{   \left\lbrace\substack{   h\in G^F\ \text{s.t.} \\  s\in {^h(T_1)^F}  } \right\rbrace   } \sum_{\tau\in {^h(T^1)^F} }         {\theta}(s^h\cdot \tau^h)\cdot   Q_{{^hT},{^hU}\cap\mathrm{Stab}_G(s)^o,b}^{\mathrm{Stab}_{G}(s)^o}(u,\tau^{-1}).
\end{equation*}
Using this formula we can evaluate the values of $\mathrm{Tr}(-,R_{T,U,b}^{\theta})$ at regular semisimple elements easily: (See Theorem~\ref{thm: evaluation at regular semisimple elements})
$$\mathrm{Tr}(s,R^{\theta}_{T,U,b})=\sum_{w\in W(T)^F}(^w\theta)(s^c),$$
from which we can derive some independence properties: One of them links $R_{T,U}^{\theta}$ to certain irreducible representations constructed by G\'erardin in \cite{Gerardin1975SeriesDiscretes} (see Remark~\ref{remark: Gerardin}, Corollary~\ref{coro: coincidence at even levels}, and Remark~\ref{remark: coincidence at odd levels}). In the final section we turn to the gamma functions $\gamma_{G^F}(\chi,\psi)$ defined with respect to characters $\chi$ of $G^F$ and $\overline{\mathbb{Q}}_{\ell}$-valued functions $\psi$ on $G^F$ (see Section~\ref{sec: gamma}); these gamma functions were introduced in \cite{Braverman_Kazhdan_2000_GammaFunc} and \cite{Braverman_Kazhdan_2000_gammaSh}, and studied based on properties of character sheaves. We show that (see Proposition~\ref{prop: gamma function})
$$\gamma_{G^F}\left(\mathrm{Tr}({R}_{T,U,b}^{\theta}),\psi\right)=\gamma_{(T_1)^F}(\theta,\psi|_{(T_1)^F}),$$
when $\theta$ and $\psi$ are $p$-constant.

\vspace{2mm} We record here some further notation and conventions. The set of roots of $\mathbf{T}$ will be denoted by $\Phi$; for $\alpha\in \Phi$, we let $\mathbf{U}_{\alpha}\subseteq\mathbf{U}$ be the corresponding root subgroup, and we let $U_{\alpha} \subseteq U$ be the corresponding algebraic group. There is an isomorphism of finite groups $N_G(T)/T\cong N_{G_1}(T_1)/T_1$ (see \cite[Section~2]{ChenStasinski_2016_algebraisation}), thus we will use $W(T)$ to denote the Weyl group of $\mathbf{T}$. We will use the conjugation notation $a^b={^{b^{-1}}a}=b^{-1}ab$ for suitable objects $a$ and $b$. For a character of a finite abelian group, we sometimes use the same notation for its representation space and itself.

\vspace{2mm}\noindent {\bf Acknowledgement.} The author thanks Alexander Stasinski for useful comments on an earlier version of this paper. This paper comes as a by-product of the author's study on Brou\'{e}--Michel's work \cite{Broue_Michel_blocs_CrelleJ}. During the preparation of this work, the author is partially supported by the STU funding NTF17021.

\section{The character formula}\label{section: character formula}

We will work with the general $R_{T,U,b}^{\theta}$ (see Section~\ref{section:Intro}) throughout this paper; a rewarding result for using this generality will be seen at Section~\ref{section: Regular semisimple elements}. 

\vspace{2mm} We want to prove the following character formula:

\begin{thm}\label{theorem: character formula}
Let $\mathfrak{R}_{T,U,b}^{\theta}$ be the character of $R_{T,U,b}^{\theta}$, then for $g\in G^F$ we have
\begin{equation*}
\begin{split}
\mathfrak{R}_{T,U,b}^{\theta}(g)=
&\frac{1}{|T^F|}\cdot \frac{1}{|(\mathrm{Stab}_{G}(s)^o)^F|} \cdot  \sum_{    \left\lbrace \substack{ h\in G^F\  \text{s.t.} \\ s\in {^h(T_1)^F} } \right\rbrace    } \sum_{\tau\in {^h(T^1)}^F}         {\theta}(s^h\cdot \tau^h) \\
&  \cdot  \mathrm{Tr}\left((u,\tau^{-1})  \mid \sum_{i}(-1)^i\cdot H_c^i(\mathrm{Stab}_G(s)^o\cap L^{-1}( {^h FU^{r-b,b} }),\overline{\mathbb{Q}}_{\ell})  \right),
\end{split}
\end{equation*}
where $g=su$ is the Jordan decomposition.
\end{thm}

\begin{proof}
By Brou\'e's character formula on bimodule induction (see e.g.\ \cite[Chapter~4]{DM1991}) we have
\begin{equation*}
\mathfrak{R}_{T,U,b}^{\theta}(g)=\frac{1}{|T^F|}\cdot \sum_{t\in T^F}\theta(t^{-1})\cdot \mathrm{Tr}\left((g,t)  \mid \sum_{i}(-1)^iH_c^i(L^{-1}(FU^{r-b,b}),\overline{\mathbb{Q}}_{\ell})  \right).
\end{equation*}
So, applying Deligne and Lusztig's fixed point formula (see \cite[3]{DL1976}) to the RHS we get
\begin{equation*}
\mathfrak{R}_{T,U,b}^{\theta}(g)=\frac{1}{|T^F|}\cdot \sum_{t\in T^F}\theta(t^{-1})\cdot \mathrm{Tr}\left((u,t'')  \mid \sum_{i}(-1)^iH_c^i(L^{-1}(FU^{r-b,b})^{(s,t')},\overline{\mathbb{Q}}_{\ell})  \right),
\end{equation*}
where $t=t't''$ is the natural decomposition via $T^F \cong (T_1)^F\times (T^1)^F$. We want to analyse the structure of the variety $L^{-1}(FU^{r-b,b})^{(s,t')}$.

\vspace{2mm} If $h\in G^F$ conjugates $t'$ to $s^{-1}$ (i.e.\ $ht'h^{-1}=s^{-1}$), and $z\in \mathrm{Stab}_G(t')$ has its Lang image belonging to $FU^{r-b,b}$ (i.e.\ $L(z)\in FU^{r-b,b}$), then clearly $hz\in L^{-1}(FU^{r-b,b})^{(s,t')}$. So we have a well-defined multiplication morphism 
$$\phi\colon \{ h\in G^F\mid ht'=s^{-1}h  \}\times \{ z\in \mathrm{Stab}_G(t')^o \mid L(z)\in FU^{r-b,b} \} \rightarrow L^{-1}(FU^{r-b,b})^{(s,t')}$$
given by $\phi(h,z)=hz$.

\vspace{2mm} We shall show that $\phi$ is surjective. Take $x\in L^{-1}(FU^{r-b,b})^{(s,t')}$, then by the rationality of $s$ and $t'$ we see that 
$$sF(x)t'=F(x)=x\cdot L(x)=sxt'\cdot L(x),$$
which implies $L(x)\cdot t'=t'\cdot L(x)$, namely $L(x)\in \mathrm{Stab}_G(t')$. Now write $L(x)=x_1 x_2$ with $x_1\in G_1$ and $x_2\in G^1$, via the product decomposition $G=G_1\ltimes G^1$, then by taking the reduction map we see $x_1\in \mathrm{Stab}_{G_1}(t')$, which also implies that $x_2\in \mathrm{Stab}_{G^1}(t')$. Since $x_1\in FU$ is unipotent, we have (see e.g.\ \cite[2.5]{DM1991})
\begin{equation}\label{temp 1}
x_1\in \mathrm{Stab}_{G_1}(t')^o.
\end{equation}
On the other hand, for a given $\tilde{x}\in \mathrm{Stab}_{G^1}(t')$, consider the unique Iwahori decomposition $\tilde{x}=\tilde{t}\tilde{u}$ (in the sense of \cite[2.2]{Sta2009Unramified}), where $\tilde{t}\in T^1$ and $\tilde{u}\in U^1(U^{-})^1$. Clearly $\tilde{t}$ commutes with $t'$, so $\tilde{u}$ also commutes with $t'$. Write $\tilde{u}=\prod u_{\alpha}$, where $\alpha$ runs over the roots and $u_\alpha\in (U_{\alpha})^1$, then the commutativity between $t'$ and $\tilde{u}$ implies that, for each root $\alpha$, either $u_{\alpha}=1$ or $\alpha(t')=1$. Therefore $\mathrm{Stab}_{G^1}(t')$ is an affine space. (Moreover, the argument implies that $\mathrm{Stab}_{G}(t')^o$ is the Weil restriction of the base change of the connected reductive group $\mathrm{Stab}_{G_1}(t')^o$ from $\overline{\mathbb{F}}_q$ to $\mathcal{O}^{\mathrm{ur}}_r$; see also the argument in \cite[2.3]{DM1991}.) In particular, we have
\begin{equation}\label{temp 2}
x_2\in \mathrm{Stab}_{G^1}(t')^o.
\end{equation}
It follows from \eqref{temp 1} and \eqref{temp 2} that 
$$L(x)\in \mathrm{Stab}_{G}(t')^o,$$ 
so by the Lang--Steinberg theorem we see $L(x)=L(z)$ for some $z\in \mathrm{Stab}_{G}(t')^o$. Let $h:=xz^{-1}$, then $h\in G^F$ and 
$$ht'h^{-1}=xz^{-1}t'zt'^{-1}x^{-1}s^{-1}=s^{-1},$$
thus $\phi$ is surjective.

\vspace{2mm} Meanwhile, note that $\phi(h,z)=\phi(h',z')$ if and only if $h^{-1}h'=zz'^{-1}\in (\mathrm{Stab}_{G}(t')^o)^F$, so, for a fixed set of representatives of $G^F/(\mathrm{Stab}_{G}(t')^o)^F$, $\phi$ induces an isomorphism
\begin{equation*}
L^{-1}(FU^{r-b,b})^{(s,t')}\cong \coprod_{ \left\lbrace\substack{   h\in G^F/(\mathrm{Stab}_{G}(t')^o)^F \\ \text{s.t.}\  ht'=s^{-1}h  } \right\rbrace  } \{ z\in \mathrm{Stab}_G(t')^o \mid L(z)\in FU^{r-b,b} \}_h,
\end{equation*}
where $\{ z\in \mathrm{Stab}_G(t')^o \mid L(z)\in FU^{r-b,b} \}_h$ is a copy of $\{ z\in \mathrm{Stab}_G(t')^o \mid L(z)\in FU^{r-b,b} \}$ on which the action of $(u,t'')$ is given by $z\mapsto (u^h)\cdot z \cdot t''$.

\vspace{2mm} Note that 
$$\{ z\in \mathrm{Stab}_G(t')^o \mid L(z)\in FU^{r-b,b} \}= \mathrm{Stab}_G(t')^o\cap L^{-1}(\mathrm{Stab}_G(t')^o\cap FU^{r-b,b})$$ 
is the corresponding higher Deligne--Lusztig variety of $\mathrm{Stab}_G(t')^o$; let us denote it by $L_{t'}$. Therefore
\begin{equation*}
\begin{split}
\mathfrak{R}_{T,U,b}^{\theta}(g)
&=\frac{1}{|T^F|}\cdot \sum_{t\in T^F}\theta(t^{-1})\cdot  \sum_{ \left\lbrace\substack{   h\in G^F/(\mathrm{Stab}_{G}(t')^o)^F \\ \text{s.t.}\  ht'=s^{-1}h  } \right\rbrace } \mathrm{Tr}\left((u^h,t'')  \mid H_c^*(L_{t'})  \right)\\
&=\frac{1}{|T^F|}\cdot \sum_{t\in T^F}\theta(t^{-1})\cdot    \sum_{ \left\lbrace\substack{   h\in G^F\ \text{s.t.} \\  ht'=s^{-1}h  } \right\rbrace   }  \frac{1}{|(\mathrm{Stab}_{G}(t')^o)^F|}\cdot  \mathrm{Tr}\left((u^h,t'')  \mid H_c^*(L_{t'})  \right)\\
&=\frac{1}{|T^F|}\cdot \frac{1}{|(\mathrm{Stab}_{G}(s)^o)^F|} \cdot  \sum_{t\in T^F}  \sum_{ \left\lbrace\substack{   h\in G^F\ \text{s.t.} \\  ht'=s^{-1}h  } \right\rbrace    }       \theta(t^{-1})\cdot   \mathrm{Tr}\left((u^h,t'')  \mid H_c^*(L_{t'})  \right)\\
&=\frac{1}{|T^F|}\cdot \frac{1}{|(\mathrm{Stab}_{G}(s)^o)^F|} \cdot  \sum_{  \left\lbrace\substack{   h\in G^F\ \text{s.t.} \\  s\in {^h(T_1)^F}  } \right\rbrace     } \sum_{\tau\in {^h(T^1)}^F}         {\theta}((s\tau)^h)\cdot   \mathrm{Tr}\left((u,\tau^{-1})  \mid H_c^*({^hL_{t'}})  \right),
\end{split}
\end{equation*}
where $H_c^*(-):=\sum_i(-1)^iH_c^i(-,\overline{\mathbb{Q}}_{\ell})$ and in the last summation we put $t'=(s^h)^{-1}$. Since $\mathrm{Stab}_G(t')=\mathrm{Stab}_G(t'^{-1})$, we complete the proof by expressing $L_{t'}$ in terms of $s$.
\end{proof}

\begin{remark}\label{remark: stabiliser of ss element}
Let $s\in G_1$ be a semisimple element, then $\mathrm{Stab}_{G_1}(s)^o$ is a connected reductive group; from the proof of Theorem~\ref{theorem: character formula} we see that an analogue of this property also holds for $\mathrm{Stab}_{G}(s)^o$: If $s\in T_1$ is a semisimple element, and if $H$ is the $\overline{\mathbb{F}}_q$-Weil restriction of $\mathrm{Stab}_{G_1}(s)^o\times_{\overline{\mathbb{F}}_q}\mathcal{O}^{\mathrm{ur}}_r$, then $H=\mathrm{Stab}_{G}(s)^o$. In particular, if $s\in T_1^F$ is rational, then we can talk about the higher Deligne--Lusztig representations of $(\mathrm{Stab}_{G}(s)^o)^F$. 
\end{remark}

\begin{remark}\label{remark: caution on Weil restriction}
In general, if $s'\in G$ is a semisimple element, then it is contained in a maximal torus, thus can be conjugated to be some $s\in T_1\subseteq G_1$ (though, there are lots of semisimple elements not in $G_1$). However, in this general situation it can happen that the property in Remark~\ref{remark: stabiliser of ss element} fails for $\mathrm{Stab}_G(s')^o$: Suppose $s'^h=s$, then the Weil restriction of $(\mathrm{Stab}_G(s')^o)_1\times_{\mathrm{Spec}~\overline{\mathbb{F}}_q} \mathrm{Spec}~\mathcal{O}^{\mathrm{ur}}_r$ is $(\mathrm{Stab}_G(s)^o)^{\rho_{r,1}(h)}$, which may not be equal to $(\mathrm{Stab}_G(s)^o)^{h}=\mathrm{Stab}_G(s')^o$.
\end{remark}

\section{Green functions}\label{section: Green functions}

We want to give a Green function theoretic interpretation of Theorem~\ref{theorem: character formula}, in a way similar to \cite[4.2]{DL1976}; we start with the following definition.

\begin{defi}\label{definition: Green function}
Suppose that $G$ is a closed subgroup of some linear algebraic group $\widetilde{G}$ and that the Frobenius endomorphism $F$ is also defined on $\widetilde{G}$. Then for $h\in \widetilde{G}^F$ and a quadruple $(G,T,U,b)$, the function 
$$Q_{T^h,U^h,b}^{G^h}\colon (\mathcal{U}_{G^h})^F\times ((T^1)^h)^F\rightarrow \overline{\mathbb{Q}}_{\ell}$$ defined by
$$(u,\tau)\mapsto  \frac{1}{|T^F|}\mathrm{Tr}\left((u,\tau)\mid  \sum_{i}(-1)^i\cdot H_c^i(L^{-1}(FU^{r-b,b})^h,\overline{\mathbb{Q}}_{\ell}) \right)$$
is called a Green function on $(G^h)^F$. Here $\mathcal{U}_{G^h}$ is the variety of unipotent elements in $G^h$.
\end{defi}

The above definition actually includes a wider class of groups than the Weil restrictions of reductive groups over $\mathcal{O}_r$, as it may happen that $G^h$ does not equal to the Weil restriction of $(G^h)_1\times_{\mathrm{Spec}~\overline{\mathbb{F}}_q} \mathrm{Spec}~\mathcal{O}^{\mathrm{ur}}_r$ (see Remark~\ref{remark: caution on Weil restriction}).

\begin{remark}
Recall that, if $r=1$, then the values of $\mathfrak{R}_{T,U}^{\theta}$ at unipotent elements are independent of $\theta$; however, this is usually not true for a general $r\geq1$ (see \cite[Section~3]{Lusztig2004RepsFinRings}). So, instead of using $R_{T,U}^{1}$ to define Green functions as in Deligne--Lusztig's original work \cite{DL1976}, we used the above average form. 
\end{remark}

Using the above Green functions, Theorem~\ref{theorem: character formula} can be re-written as:

\begin{thm}\label{theorem: character formula 2}
Let $\mathfrak{R}_{T,U,b}^{\theta}$ be the character of $R_{T,U,b}^{\theta}$, then for $g\in G^F$ we have
\begin{equation*}
\mathfrak{R}_{T,U,b}^{\theta}(g)=\frac{1}{|(\mathrm{Stab}_{G}(s)^o)^F|} \cdot  \sum_{   \left\lbrace\substack{   h\in G^F\ \text{s.t.} \\  s\in {^h(T_1)^F}  } \right\rbrace   } \sum_{\tau\in {^h(T^1)^F} }         {\theta}(s^h\cdot \tau^h)\cdot   Q_{{^hT},{^hU}\cap\mathrm{Stab}_G(s)^o,b}^{\mathrm{Stab}_{G}(s)^o}(u,\tau^{-1}),
\end{equation*}
where $g=su$ is the Jordan decomposition.
\end{thm}

In the $r=1$ case, the Green functions are $\mathbb{Z}$-valued (see e.g.\ \cite[7.6]{Carter1993FiGrLieTy}); a similar property holds for a general $r$.

\begin{coro}\label{coro: integer values of Green functions}
The function $\sum_{ \tau\in (T^1)^F }Q_{T,U,b}^G(-,\tau)$ on $(\mathcal{U}_G)^F$ is $\mathbb{Z}$-valued.
\end{coro}
\begin{proof}
Let $u\in G^F$ be a unipotent element. By Theorem~\ref{theorem: character formula 2} we have
\begin{equation*}
\begin{split}
 \mathfrak{R}^{1}_{T,U,b}(u)
&=\frac{1}{|G^F|} \sum_{h\in G^F} \sum_{\tau\in {^h(T^1)^F}}  Q^G_{{^hT},{^hU},b}(u,\tau^{-1})\\
&=\frac{1}{|G^F|} \sum_{h\in G^F} \sum_{\tau\in {(T^1)^F}}  Q^G_{{T},{U},b}(u,\tau^{-1})=\sum_{ \tau\in (T^1)^F }Q_{T,U,b}^G(u,\tau),
\end{split}
\end{equation*}
where the second equality follows from that characters are class functions and $h\in G^F$. 

\vspace{2mm} In particular, $\sum_{ \tau\in (T^1)^F }Q_{T,U,b}^G(-,\tau)$ has its values being algebraic integers. Meanwhile, by basic properties of Lefschetz numbers (see \cite[1.2]{Lusztig_whiteBk}) we know this sum takes values in $\mathbb{Q}$, so it must takes values in $\mathbb{Z}$.
\end{proof}

\begin{defi}
Let $p:=\mathrm{char}(\mathbb{F}_q)$. A class function $f\colon G^F\rightarrow \overline{\mathbb{Q}}_{\ell}$ is called $p$-constant if $f(g)=f(s)$ for all $g\in G^F$, where $g=su$ is the Jordan decomposition with $s$ semisimple and $u$ unipotent.
\end{defi}

The notion of $p$-constant functions is very useful in the representations of Lie type groups. (For example, see \cite[Section~2]{Broue_Michel_blocs_CrelleJ} for an application on the relations between $\ell$-blocks and Lusztig series.) The following corollary generalises a property in the $r=1$ case (see \cite[3.8]{Digne_Michel_FoncteursLusztig_J.Alg_1987}).

\begin{coro}\label{coro: a product property}
Let $f\colon G^F\rightarrow \overline{\mathbb{Q}}_{\ell}$ be a $p$-constant class function, then $\mathfrak{R}_{T,U,b}^{\theta\cdot \mathrm{Res}_{T^F}^{G^F}f}=\mathfrak{R}_{T,U,b}^{\theta}\cdot f$. 
\end{coro}
\begin{proof}
Let $g\in G^F$ and let $g=su$ be the Jordan decomposition. By Theorem~\ref{theorem: character formula 2} we see that $\mathfrak{R}_{T,U,b}^{\theta\cdot \mathrm{Res}_{T^F}^{G^F}f}(g)$ equals to
$$\frac{1}{|(\mathrm{Stab}_{G}(s)^o)^F|} \cdot  \sum_{  \left\lbrace\substack{   h\in G^F\ \text{s.t.} \\  s\in {^h(T_1)^F}  } \right\rbrace   } \sum_{\tau\in {^h(T^1)^F} }         {\theta}((s\tau)^h)\cdot {f}((s\tau)^h)\cdot   Q_{{^hT},{^hU}\cap\mathrm{Stab}_G(s)^o,b}^{\mathrm{Stab}_{G}(s)^o}(u,\tau^{-1}).$$
Note that $f((s\tau)^h)=f(s\tau)=f(s)=f(g)$, which completes the proof.
\end{proof}

There is also an inner product version of the above corollary (see \cite[7.11]{DL1976} or \cite[7.6.3]{Carter1993FiGrLieTy} for the $r=1$ case).

\begin{coro}\label{coro: a product property in inner product form}
Let $R$ be the representation space of a $p$-constant virtual character of $G^F$, then $\langle R,R_{T,U,b}^{\theta}\rangle_{G^F}=\langle \mathrm{Res}^{G^F}_{T^F},\theta\rangle_{T^F}$. 
\end{coro}
\begin{proof}
By Corollary~\ref{coro: a product property} and the Hom--tensor adjunction we see
\begin{equation}\label{temp 4}
\langle R,R_{T,U,b}^{\theta}\rangle_{G^F}=\langle R_{T,U,b}^{\theta^{-1}\cdot \mathrm{Res}_{T^F}^{G^F}\chi_R},1\rangle_{G^F},
\end{equation}
where $\chi_R$ means the character of $R$. As $R^{(-)}_{T,U,b}$ is the induction (from the virtual representations of $T^F$ to the virtual representations of $G^F$) provided by the alternating sum of bimodules $\sum_i(-1)^iH_c^i(L^{-1}(FU^{b,r-b}),\overline{\mathbb{Q}}_{\ell})$, there is a restriction functor ${^*R_{T,U,b}^{(-)}}$ adjoint to $R_{T,U,b}^{(-)}$ (see \cite[Chapter~4]{DM1991} for more details); we get
$$\eqref{temp 4}=\langle \theta^{-1}\otimes \mathrm{Res}_{T^F}^{G^F}\chi_R,{^*R_{T,U,b}^{1}}\rangle_{T^F}.$$
However, by basic properties of the $\ell$-adic cohomology we see that
$${^*R_{T,U,b}^{1}}=\sum_i(-1)^iH_c^{i}(G^F\backslash L^{-1}(FU^{b,r-b}),\overline{\mathbb{Q}}_{\ell})=\sum_i(-1)^iH_c^{i}(FU^{b,r-b},\overline{\mathbb{Q}}_{\ell})=\overline{\mathbb{Q}}_{\ell}$$ 
is the trivial representation (\cite[6.4]{DL1976}), from which the assertion follows.
\end{proof}

There is a ``Green (function) integration formula'':

\begin{prop}\label{prop: summation formula}
Let $\widetilde{G}$ be as in Definition~\ref{definition: Green function}. We have 
$$\sum_{u\in (\mathcal{U}_{G^x})^F}\sum_{\tau\in ((T^1)^x)^F} Q_{T^x,U^x,b}^{G^x}(u,\tau)=|G^F/T_1^F|$$ 
for every $x\in \widetilde{G}^F$. (Compare the $r=1$ case in \cite[7.6.1]{Carter1993FiGrLieTy}.)
\end{prop}

\begin{proof}
We use an induction argument on $i:=|\dim G/\dim T|$. If $i=1$, then $G=T$, hence
$$\sum_{u\in (\mathcal{U}_{G^x})^F}\sum_{\tau\in ((T^1)^x)^F} Q_{T^x,U^x,b}^{G^x}(u,\tau)=\frac{1}{|T^F|}\sum_{u\in ((T^1)^x)^F}\sum_{\tau\in ((T^1)^x)^F} \mathrm{Tr}((u,\tau) \mid \overline{\mathbb{Q}}_{\ell}[(T^x)^F] )=|T^F/T_1^F|,$$
as desired. Suppose now the assertion is true for every $i\leq n$, and suppose $\dim G/\dim T=n+1$. From the proof of Corollary~\ref{coro: a product property in inner product form} we see that $\langle \sum_i(-1)^i H_c^i(L^{-1}(FU)^x,\overline{\mathbb{Q}}_{\ell}),1 \rangle_{(G^x)^F}=1$, that is: (Use Theorem~\ref{theorem: character formula 2})
\begin{equation}\label{temp 5}
\begin{split}
\frac{1}{|G^F|} &  \sum_{g=su\in (G^x)^F}\frac{1}{|(\mathrm{Stab}_{G^x}(s)^o)^F|}  \\
& \times  \sum_{   \left\lbrace\substack{   h\in (G^x)^F\ \text{s.t.} \\  s\in {^h((T_1)^x)^F}  } \right\rbrace    }    \sum_{\tau\in {^h((T^1)^x)^F} }          Q_{{^h(T^x)},{^h(U^x)}\cap\mathrm{Stab}_{G^x}(s)^o,b}^{\mathrm{Stab}_{G^x}(s)^o}(u,\tau^{-1})=1,
\end{split}
\end{equation}
where $g=su$ denotes the Jordan decomposition. 

\vspace{2mm} To proceed, we need to show that $u\in \mathrm{Stab}_{G^x}(s)^o$ for the Jordan decomposition in the above summation, which is well-known when $r=1$. By construction there is a $y\in \widetilde{G}^F$ such that $s^y\in (T_1)^F$ and $G^{xy}=G$. So $\rho_{r,1}(u^y)\in \mathrm{Stab}_{G}(s^y)^o$  (use e.g.\ Remark~\ref{remark: stabiliser of ss element}), thus $u^y\in \mathrm{Stab}_{G}(s^y)^o$, so $u\in \mathrm{Stab}_{G^x}(s)^o$.  Therefore \eqref{temp 5} becomes
\begin{equation}\label{temp 6}
\begin{split}
\frac{1}{|G^F|} &  \sum_{   \substack{     s\in (G^x)^F \\   \text{semisimple}  }          }    \frac{1}{|(\mathrm{Stab}_{G^x}(s)^o)^F|} \sum_{    \substack{    u\in (\mathrm{Stab}_{G^x}(s)^o)^F \\   \text{unipotent} }     } \\
& \times \sum_{   \left\lbrace\substack{   h\in (G^x)^F\ \text{s.t.} \\  s\in {^h((T_1)^x)^F}  } \right\rbrace     }    \sum_{\tau\in {^h((T^1)^x)^F} }          Q_{{^h(T^x)},{^h(U^x)}\cap\mathrm{Stab}_{G^x}(s)^o,b}^{\mathrm{Stab}_{G^x}(s)^o}(u,\tau^{-1})=1.
\end{split}
\end{equation}
There are two cases of $s$, depending on whether $s\in Z^x$, where $Z$ denotes the center of $G$. As $\mathrm{Stab}_{G^x}(s)^o=G^x$ if and only if $s\in Z^x$, by our induction assumption we can re-write \eqref{temp 6} as (note that $h\in (G^x)^F$)
\begin{equation*}
\begin{split}
\frac{|Z_1^F|}{|G^F|} \cdot &  \sum_{u\in (\mathcal{U}_{G^x})^F  }    \sum_{\tau\in {((T^1)^x)^F} }          Q_{{T^x},{U^x},b}^{G^x}(u,\tau^{-1}) \\
& + \frac{1}{|G^F|}  \sum_{   \substack{  s\in (G^x)^F\backslash (Z^x)^F  \\  \text{semisimple} }  } \sum_{    \left\lbrace  \substack{   h\in (G^x)^F\ \text{s.t.} \\ s\in {^h((T_1)^x)^F}    }   \right\rbrace   }  1/|T_1^F|  =1.
\end{split}
\end{equation*}
Therefore
\begin{equation}\label{temp 7}
\begin{split}
\frac{|Z_1^F|}{|G^F|} \cdot  & \sum_{u\in (\mathcal{U}_{G^x})^F  }   \sum_{\tau\in {((T^1)^x)^F} }          Q_{{T^x},{U^x},b}^{G^x}(u,\tau^{-1}) -\frac{1}{|G^F|}  \sum_{s\in  (Z^x)_1^F} \sum_{ \left\lbrace  \substack{   h\in (G^x)^F\ \text{s.t.} \\ s\in {^h((T_1)^x)^F}    }   \right\rbrace    }  1/|T_1^F| \\
& + \frac{1}{|G^F|}  \sum_{   \substack{  s\in (G^x)^F \\ \text{semisimple}  }  } \sum_{ \left\lbrace  \substack{   h\in (G^x)^F\ \text{s.t.} \\ s\in {^h((T_1)^x)^F}    }   \right\rbrace    }  1/|T_1^F|  =1.
\end{split}
\end{equation}
In the index of the above summation, note that, if $s\in {^h((T_1)^x)^F}$, then $s$ is automatically semisimple, thus \eqref{temp 7} becomes
\begin{equation*}
\begin{split}
\frac{|Z_1^F|}{|G^F|} \cdot  & \sum_{u\in (\mathcal{U}_{G^x})^F  }   \sum_{\tau\in {((T^1)^x)^F} }          Q_{{T^x},{U^x},b}^{G^x}(u,\tau^{-1}) -\frac{|Z_1^F|}{|T_1^F|} +1  =1,
\end{split}
\end{equation*}
from which the proposition follows.
\end{proof}

\section{Regular semisimple elements}\label{section: Regular semisimple elements}

In this section we focus on the values of higher Deligne--Lusztig characters at regular semisimple elements. Recall that a semisimple element is called regular, if its centraliser is of minimal dimension; in our situation, this is equivalent to the saying that, a semisimple element $s$ is regular if and only if $\mathrm{Stab}_G(s)^o$ is isomorphic to $T$.

\begin{thm}\label{thm: evaluation at regular semisimple elements}
Let $s\in G^F$ be a regular semisimple element. We have: $\mathfrak{R}^{\theta}_{T,U,b}(s)=0$ if the conjugacy class of $s$ in $G^F$ does not intersect $T_1$, and
$$\mathfrak{R}^{\theta}_{T,U,b}(s)=\sum_{w\in W(T)^F}(^w\theta)(s^c)$$
if $s^c\in T_1$ for some $c\in G^F$ (note that this can happen even if $s\in G\backslash G_1$). Here $W(T):=N(T)/T$ is isomorphic to the Weyl group $W(T_1):=N_{G_1}(T_1)/T_1$; see \cite[Section~2]{ChenStasinski_2016_algebraisation}.
\end{thm}

\begin{proof}
From the character formula we only need to deal with the case that the intersection of $T_1^F$ and the conjugacy class of $s$ (in $G^F$) is non-empty. 

\vspace{2mm} By regularity, the only $G$-conjugation of $T_1$ containing $s^c$ is $T_1$ itself; see e.g.\ \cite[Section~2.3]{Humphreys_1995_Conj_ss_algGp}. (Also note that the only conjugation of $T$ containing $s^c$ is $T$, as otherwise $\mathrm{Stab}_G(s^c)^o$ is not isomorphic to $T$.) So the formula in Theorem~\ref{theorem: character formula 2} can be simplified as
\begin{equation}\label{temp 3}
\mathfrak{R}_{T,U,b}^{\theta}(s)=\mathfrak{R}_{T,U,b}^{\theta}(s^c)=\frac{1}{|T^F|} \cdot  \sum_{ h\in N_G(T_1)^F  } \sum_{\tau\in (T^1)^F}         {\theta}(s^{c\cdot h}\cdot \tau^{h})\cdot   Q_{{T},\{1\},b}^{T}(1,\tau^{-1}).
\end{equation}
In this summation, note that the function 
$$Q_{{T},\{1\},b}^{T}(1,\tau^{-1})=\frac{1}{|T^F|}\mathrm{Tr}\left((1,\tau^{-1})  \mid\overline{\mathbb{Q}}_{\ell}[T^F]  \right)$$ 
on $(T^1)^F$ is the characteristic function at the identity element, and note that $N_G(T_1)=N_{G_1}(T_1)T^1$ by an Iwahori decomposition argument as in the proof of Theorem~\ref{theorem: character formula}, thus
\begin{equation*}
\eqref{temp 3}=\sum_{w\in W(T)^F}(^w\theta)(s^c),
\end{equation*}
as desired.
\end{proof}

\begin{coro}\label{coro: evaluation independence at regular ss elements}
The values of $\mathfrak{R}_{T,U,b}^{\theta}$ at regular semisimple elements are independent of the choices of $U$ and $b$. 
\end{coro}
\begin{proof}
This follows immediately from Theorem~\ref{thm: evaluation at regular semisimple elements}.
\end{proof}

\begin{remark}\label{remark: Gerardin}
In \cite{Gerardin1975SeriesDiscretes}, by purely algebraic methods, whenever $(G,T,\theta)$ satisfying certain conditions (namely, $\mathbb{G}$ is defined from an unramified split group with the derived subgroup being simply connected, $T$ is special in the sense of \cite[3.3.9]{Gerardin1975SeriesDiscretes}, and $\theta$ is regular and in general position in the sense of \cite{Lusztig2004RepsFinRings}; see \cite[Remark~3.4]{ChenStasinski_2016_algebraisation}), G\'erardin constructed an irreducible representation $R(\theta)$ of $G^F$ of the form: If $r$ is even, then $R(\theta)=\mathrm{Ind}_{(TG^{r/2})^F}^{G^F}\widetilde{\theta}$, where $\widetilde{\theta}$ is the trivial lift of $\theta$; if $r$ is odd, then $R(\theta)=\mathrm{Ind}_{(TG^{(r-1)/2})^F}^{G^F}\widetilde{\theta}$, where $\widetilde{\theta}$ is some irreducible representation of $(TG^{(r-1)/2})^F$ of dimension $q^{\# \Phi/2}$. Lusztig suggested in \cite{Lusztig2004RepsFinRings} that these algebraically constructed representations are likely to be the same as the geometrically constructed representations $R_{T,U}^{\theta}$; when $r$ is even, this was proved in \cite{ZheChen_PhDthesis} for $\mathrm{GL}_n$ and in \cite{ChenStasinski_2016_algebraisation} in general. 
\end{remark}

However, even when $r$ is even, if one does not impose any restrictions on $\theta$, then, as can be seen from Lusztig's computations \cite[Section~3]{Lusztig2004RepsFinRings}, it can happen that $R_{T,U}^{\theta}$ is not isomorphic to $\mathrm{Ind}_{(TG^{r/2})^F}^{G^F}\widetilde{\theta}$ for some $\theta$; in any case, we show that, even though the characters of these two representations may not be identical, they always agree at regular semisimple elements, without any conditions on $\theta$:

\begin{coro}\label{coro: coincidence at even levels}
When $r$ is even, $\mathfrak{R}_{T,U}^{\theta}(s)=\mathrm{Tr}(s, \mathrm{Ind}_{(TG^{r/2})^F}^{G^F}\widetilde{\theta})$ for any regular semisimple element $s\in G^F$.
\end{coro}
\begin{proof}
Let $b=r/2$, then $R_{T,U,b}^{\theta}\cong \mathrm{Ind}_{(TG^{r/2})^F}^{G^F}\widetilde{\theta}$ according to \cite[3.3]{ChenStasinski_2016_algebraisation}. So the assertion follows from  Corollary~\ref{coro: evaluation independence at regular ss elements}.
\end{proof}

\begin{remark}\label{remark: coincidence at odd levels}
If $(G,T,\theta)$ satisfies G\'erardin's conditions mentioned in Remark~\ref{remark: Gerardin}, then the above equality also holds for $r$ odd, i.e.\ $\mathfrak{R}_{T,U}^{\theta}(s)=\mathrm{Tr}(s, R(\theta))$ for any regular semisimple element $s\in T^F$ and any $r$: Indeed, when $r$ is odd, according to \cite[4.3.4]{Gerardin1975SeriesDiscretes}, the character value of $R(\theta)$ at $s$ is $\frac{1}{|(TG^{(r-1)/2})^F|}\sum_{\{  h\in G^F\mid s^h\in (TG^{(r-1)/2})^F \}}\theta(s^h)=\sum_{w\in W(T)^F} \theta(s^w)$.
\end{remark}

Theorem~\ref{thm: evaluation at regular semisimple elements} also implies an agreement at regular semisimple elements in another direction: It is easy to see that the image of a reduction map on a Deligne--Lusztig variety is a Deligne--Lusztig variety at a lower level (see e.g.\ \cite[Lemma~3.3.3]{ZheChen_PhDthesis}). In particular, the image of $L^{-1}(FU)$ along $\rho_{r,1}$ is a classical Deligne--Lusztig variety; let $\mathfrak{R}_{T_1}^{\theta_1}$ be its associated Deligne--Lusztig character, where $\theta_1$ denotes the restriction of $\theta$ to $T_1^F$.

\begin{coro}\label{coro: coincidence to level one}
Let $s$ be a regular semisimple element of $G_1^F$, then $\mathfrak{R}_{T,U}^{\theta}(s)=\mathfrak{R}_{T_1}^{\theta_1}(s)$.
\end{coro}
\begin{proof}
This follows immediately from Theorem~\ref{thm: evaluation at regular semisimple elements}.
\end{proof}

\begin{remark}
Note that the assertions in Theorem~\ref{thm: evaluation at regular semisimple elements}, Corollary~\ref{coro: evaluation independence at regular ss elements} Corollary~\ref{coro: coincidence at even levels}, and Corollary~\ref{coro: coincidence to level one} do not hold for every semisimple $s$, in general; for example, they do not hold for $s=1$ for some $\theta$, as can be seen from \cite{Lusztig2004RepsFinRings}.
\end{remark}

\section{Gamma functions}\label{sec: gamma}

Let $\chi$ be a character of $G^F$ and let $\psi$ be a $\overline{\mathbb{Q}}_{\ell}$-valued function on $G^F$. Then there is a natural notion of gamma function on $G^F$:
$$\gamma_{G^F}(\chi,\psi):=\int_{G^F} \chi(-)\psi(-)=\frac{1}{|G^F|}\sum_{g\in G^F} \chi(g)\psi(g).$$ 
When $r=1$, these functions (up to a constant) were studied by Braverman and Kazhdan in \cite{Braverman_Kazhdan_2000_GammaFunc} and \cite{Braverman_Kazhdan_2000_gammaSh}, in which they showed that (see \cite[5.3]{Braverman_Kazhdan_2000_gammaSh}), if $\chi$ is an irreducible constituent of a Deligne--Lusztig character, and if $\psi$ is a class function defined from the trace map, then the gamma function is equal to the gamma function of the corresponding finite torus. Their method is geometric and based on deep properties of character sheaves. Here we show that the Green function summation formula (i.e.\ Proposition~\ref{prop: summation formula}) implies that the same assertion holds for higher Deligne--Lusztig characters with $p$-constant parameters, for any $r\geq1$:

\begin{prop}\label{prop: gamma function}
Let $\psi$ be a class function on $G^F$. Suppose that both $\chi$ and $\theta$ are $p$-constant, then 
$$\gamma_{G^F}\left(\mathfrak{R}_{T,U,b}^{\theta},\psi\right)=\gamma_{(T_1)^F}(\theta,\psi|_{(T_1)^F}).$$
\end{prop}

\begin{proof}
From Theorem~\ref{theorem: character formula 2} and the argument of Proposition~\ref{prop: summation formula} we see
\begin{equation*}
\begin{split}
\gamma_{G^F}\left(\mathfrak{R}_{T,U,b}^{\theta},\psi\right)=\frac{1}{|G^F|} & \sum_{ \substack{  s\in G^F  \\  \text{semisimple}  }  }  \frac{1}{|(\mathrm{Stab}_{G}(s)^o)^F|} \sum_{ \substack{  u\in (\mathrm{Stab}_{G}(s)^o)^F  \\  \text{unipotent}  }  } \psi(su) \\  
&\times \sum_{   \left\lbrace\substack{   h\in G^F\ \text{s.t.} \\  s\in {^h(T_1)^F}  } \right\rbrace   } \sum_{\tau\in {^h(T^1)^F} }         {\theta}(s^h\cdot \tau^h)\cdot   Q_{{^hT},{^hU}\cap\mathrm{Stab}_G(s)^o,b}^{\mathrm{Stab}_{G}(s)^o}(u,\tau^{-1}).
\end{split}
\end{equation*}
From the $p$-constant property we can re-write the above formula as
\begin{equation}\label{temp 8}
\begin{split}
\gamma_{G^F}\left(\mathfrak{R}_{T,U,b}^{\theta},\psi\right)=\frac{1}{|G^F|} & \sum_{ \substack{  s\in G^F  \\  \text{semisimple}  }  }  \sum_{   \left\lbrace\substack{   h\in G^F\ \text{s.t.} \\  s\in {^h(T_1)^F}  } \right\rbrace   }  \frac{1}{|(\mathrm{Stab}_{G}(s)^o)^F|} \cdot \psi(s) \cdot   {\theta}(s^h)  \\  
&\times    \sum_{ \substack{  u\in (\mathrm{Stab}_{G}(s)^o)^F  \\  \text{unipotent}  }  }    \sum_{\tau\in {^h(T^1)^F} }     Q_{{^hT},{^hU}\cap\mathrm{Stab}_G(s)^o,b}^{\mathrm{Stab}_{G}(s)^o}(u,\tau^{-1}).
\end{split}
\end{equation}
By Proposition~\ref{prop: summation formula} we see
\begin{equation*}
\begin{split}
\eqref{temp 8}
&=\frac{1}{|G^F|} \sum_{ \substack{  s\in G^F  \\  \text{semisimple}  }  }  \sum_{   \left\lbrace\substack{   h\in G^F\ \text{s.t.} \\  s\in {^h(T_1)^F}  } \right\rbrace   } \frac{1}{|(\mathrm{Stab}_{G}(s)^o)^F|} \cdot \psi(s) \cdot   {\theta}(s^h) \cdot \frac{|(\mathrm{Stab}_G(s)^o)^F|}{|T_1^F|}\\
&=\frac{1}{|T_1^F|} \cdot \frac{1}{|G^F|}\sum_{ \substack{  s\in G^F  \\  \text{semisimple}  }  }  \sum_{   \left\lbrace\substack{   h\in G^F\ \text{s.t.} \\  s\in {^h(T_1)^F}  } \right\rbrace   }  \psi(s^h) \cdot   {\theta}(s^h) \\
&=\gamma_{(T_1)^F}(\theta,\psi|_{(T_1)^F}),
\end{split}
\end{equation*}
as desired.
\end{proof}

\bibliographystyle{alpha}
\bibliography{zchenrefs}

\end{document}